\documentclass{amsart}
\usepackage[margin=1.5cm]{geometry}
\usepackage{amssymb,amsthm,comment,color,caption,enumerate,adjustbox,enumitem,tikz,
amsmath, amsfonts,amsthm,amssymb,amscd, verbatim, graphicx,color,multirow,booktabs, caption,tikz,tikz-cd, mathdots, todonotes}
\usepackage{colortbl, xcolor,bm}
\usepackage{chngcntr, arydshln, soul, float,subcaption}
\numberwithin{equation}{section}
\usetikzlibrary{arrows,shapes,positioning,decorations.markings}
\usetikzlibrary{matrix, calc, arrows, graphs}
\usepackage{hyperref}
\newtheorem{theorem}{Theorem}[section]
\newtheorem{lemma}[theorem]{Lemma}
\newtheorem{proposition}[theorem]{Proposition}

\newtheorem{corollary}[theorem]{Corollary}

\counterwithin{table}{section}
\begin{document}
\title[Engel graphs]{On the strong connectivity of the $2$-Engel graphs of almost simple groups}

\author[F. Dalla Volta]{Francesca Dalla Volta}
\address{Francesca Dalla Volta, Dipartimento di Matematica Pura e Applicata,\newline
 University of Milano-Bicocca, Via Cozzi 55, 20126 Milano, Italy} 
\email{f.dallavolta@unimib.it}

\author[F. Mastrogiacomo]{Fabio Mastrogiacomo}
\address{Fabio Mastrogiacomo, Dipartimento di Matematica ``Felice Casorati", University of Pavia, Via Ferrata 5, 27100 Pavia, Italy} 
\email{fabio.mastrogiacomo01@universitadipavia.it}

\author[P. Spiga]{Pablo Spiga}
\address{Pablo Spiga, Dipartimento di Matematica Pura e Applicata,\newline
 University of Milano-Bicocca, Via Cozzi 55, 20126 Milano, Italy} 
\email{pablo.spiga@unimib.it}
\thanks{The authors are members of the GNSAGA INdAM research group and kindly acknowledge their support.}
	\maketitle

        \begin{abstract}
The Engel graph of a finite group $G$  is a directed graph encoding the pairs of elements in $G$ satisfying some Engel word. Recent work of Lucchini and the third author shows that, except for a few well-understood cases, the Engel graphs of almost simple groups are strongly connected. In this paper, we give a refinement to this analysis.
        	          \end{abstract}

\section{Introduction}
In this paper we investigate a directed graph introduced by Peter Cameron in~\cite[Section 11.1]{cam}. Let $x$ and $y$ be free generators of a free group of rank $2$. We let $[x,y]:=x^{-1}y^{-1}xy$.  We define recursively $[x,_0y]:=x$ and $[x,_{i+1}y]:=[[x,_{i}y],y],$ for every $i\ge 0$. The word $[x,_ny]$ is the \textit{\textbf{$n^{\mathrm{th}}$ Engel word}}.  

Now, let $G$ be a group and let  $I_n(G)=\{x\in G\mid [x,_ny]=[y,_nx]=1, {\text { for every $y\in G$}}\}$ be  the set of elements of $G$ that are  right and left $n$-Engel. The $n^{\mathrm{th}}$ \textit{\textbf{Engel graph}} $$\Gamma_{n}(G)$$ of $G$ is the directed graph having vertex set $G\setminus I_n(G)$, where $(x,y)$ is declared to be an arc if and only  if $[x,_ny]=1$. Clearly, when $n:=1$, $I_1(G)$ is the center ${\bf Z}(G)$ of $G$ and $\Gamma_1(G)$ is the \textbf{\textit{commuting graph}} of $G$.

A directed graph is \textit{\textbf{strongly connected}} if, for any two vertices, there exists a directed path from the first to the second. A directed graph is \textit{\textbf{connected}} if, for any two vertices, there exists a (not necessarily  directed) path from the first to the second.   The commuting graph is undirected because the commutator word is symmetric in $x$ and $y$, but in general $\Gamma_n(G)$ is directed. Thus, for $n\ge 2$, it is interesting to study the connectivity and the strong connectivity of $\Gamma_n(G)$.  

 There is a natural reason for excluding  the elements of $I_n(G)$ from the vertex set of $\Gamma_n(G)$: if we include elements of $I_n(G)$ in the vertex set, then we define a graph which is trivially strongly connected, because any element of $I_n(G)$ is adjacent to every other vertex. Observe that, for every $n$, $\Gamma_n(G)$ is a subgraph of $\Gamma_{n+1}(G)$; broadly speaking, this means that the family of graphs $(\Gamma_n(G))_n$ becomes denser as $n$ increases.

% We are also interested in a ``cumulative'' version of $\Gamma_n(G)$. 
%Let $1={\bf Z}_0(G)\le {\bf Z}_1(G)\le{\bf Z}_2(G)\le \cdots $ be the series of subgroups of $G$, where ${\bf Z}_{n+1}(G)/{\bf Z}_n(G)={\bf Z}(G/{\bf Z}_n(G))$. The subgroup $${\bf Z}_\infty(G):=\bigcup_{n\ge 0}{\bf Z}_n(G)$$ is called the \textit{\textbf{hypercenter}} of $G$. The Engel graph $\Gamma(G)$ of $G$ is the directed  graph having vertex set $G\setminus {\bf Z}_\infty(G)$, where $(x,y)$ is declared to be an arc if and only if $[x,_ny]=1$ for some $n\in\mathbb{N}$.
% It is shown  in~\cite[Introduction]{DLN} that the hypercenter ${\bf Z}_\infty(G)$ is the set of elements of $G$ with the property that, for every $y\in G$, there exists $n$ with $[x,_ny]=[y,_nx]=1$. As above, this is the reason for excluding  the elements of ${\bf Z}_\infty(G)$ from the vertex set of $\Gamma(G)$.

The first investigation on Engel graphs is in~\cite{DLN}. For instance,  for groups $G$ having trivial hypercenter,~\cite{DLN} reduces the study of the strong connectivity of the Engel graphs for arbitrary groups to the study of the strong connectivity of the Engel graphs of almost simple groups. The strong connectivity of the Engel graphs for almost simple groups is studied in~\cite{LS}. 

In this paper, we give a further contribution to the knowledge on the strong connectivity of the Engel graphs of almost simple groups. Since the commuting graph $\Gamma_1(G)$ is rather special in our family of Engel graphs and since it has received a lot of attention (see~\cite{mp} and the references therein), we will be focusing on the Engel graphs $\Gamma_n(G)$ with $n\ge 2$. We determine, for each almost simple group $G$, the smallest $n\ge 2$ (if it exists) such that $\Gamma_n(G)$ is strongly connected. This answers one of the questions raised in the introductory section of~\cite{LS}.

\begin{theorem}\label{main}
Let $G$ be an almost simple group with socle $L$ and let $n\ge 2$ be the smallest natural number $n$ such that $\Gamma_n(G)$ is strongly connected. Then $n=2$, except when $G$ is in Table~$\ref{table1main}$.\footnote{In Table~\ref{table1main}, we take into account the exceptional isomorphisms between non-abelian simple groups. For example, the group $\mathrm{Alt}(5)$ appears as $\mathrm{PSL}_2(4)$ or as $\mathrm{PSL}_2(5)$. Moreover, $\mathrm{PSL}_3(2)\cong\mathrm{PSL}_2(7)$.}
\end{theorem}
\begin{table}[!ht]
\begin{tabular}{c|c|c}\hline 
\toprule[1.5pt]
Group&$n$& Comments\\
\midrule[1.5pt]
$\mathrm{PSL}_2(9)\cong \mathrm{Alt}(6)$&$3$&\\
$M_{10}\cong\mathrm{PSL}_2(9).2$&3&\\
$\mathrm{PSL}_2(q)$&does not exist&$q=2^f$, $f\ge 2$, or $q\equiv 5\pmod 8$\\
$\mathrm{PSL}_2(q)$&$n$&$q\equiv 3\pmod 4$, where $2^{n-1}$\\
&& is the largest power of $2$ dividing $(q+1)/2$\\
${}^2B_2(q)$&does not exist&$q=2^{2f+1}$, $f>0$\\
$\mathrm{Aut}({}^2B_2(q))$&does not exist&$q=2^{2f+1}$, $f>0$, $2f+1$ prime\\
$\mathrm{PSL}_3(4)$&3&\\
\bottomrule[1.5pt]
\end{tabular}
\caption{Auxiliary table for Theorem~$\ref{main}$}\label{table1main}
\end{table}

%\subsection{The relevant cases}\label{sec:relevant}
To prove Theorem~\ref{main}, we use the results in~\cite{LS}. Indeed, except for a few families of almost simple groups, Theorem~\ref{main} follows  from the work in~\cite{LS}. In this section, we discuss the cases that have been proven in~\cite{LS} (and in~\cite{DLN} for alternating groups) and we list the remaining cases.
\begin{itemize}
\item[$\bm{\mathrm{Alt}(n)}$] Using~\cite{DLN}, Theorem~\ref{main} follows when $L\cong\mathrm{Alt}(n)$. Strictly speaking, the results in~\cite{DLN} are not phrased to obtain directly Theorem~\ref{main} for alternating groups, however, the proofs in~\cite{DLN} show that, except when $n\in \{5,6\}$, $\Gamma_2(\mathrm{Alt}(n))$ and $\Gamma_2(\mathrm{Sym}(n))$ are strongly connected. The veracity of Theorem~\ref{main} for almost simple groups having socle $\mathrm{Alt}(5)$ and $\mathrm{Alt}(6)$ can be checked with a computation with the aid of a computer.% In particular, there is no remaining case to be dealt with.
\item[$\bm{\mathrm{PSL}_m(q)}$] Using Theorems~3.8 and~3.11 in~\cite{LS} and Theorem~5.9 in~\cite{DLN}, Theorem~\ref{main} follows when $L\cong\mathrm{PSL}_2(q)$. From Proposition~4.2 and the proof of Corollary~1.4 in~\cite{LS}, Theorem~\ref{main} follows when $L\cong\mathrm{PSL}_m(q)$, where either $m\ge 4$, or $m=3$ and $q$ is odd. %In particular, the only remaining case is $\mathrm{PSL}_3(q)$, with $q$ even.
\item[$\bm{\mathrm{PSU}_m(q)}$] From Proposition~5.5 and the proof of Corollary~1.4 in~\cite{LS}, Theorem~\ref{main} follows when $L\cong\mathrm{PSU}_m(q)$. %Here there are no remaining cases. 
\item[$\bm{\mathrm{PSp}_m(q)}$]Theorem~\ref{main} follows from~\cite[Proposition~6.1]{LS} when $L\cong\mathrm{PSp}_m(q)$. %Here there are no remaining cases.  \item[$\bm{\mathrm{P}\Omega_m^\varepsilon(q)}$]Theorem~\ref{main} follows from~\cite[Propositions~7.1,~7.2 and~7.3]{LS} when $L\cong\mathrm{P}\Omega_{m}^\varepsilon(q)$ and $\varepsilon\in \{\circ,+,-\}$. %Here there are no remaining cases. 
\item[\bf{Exc. Grp.}]When $L$ is an exceptional group of Lie type, Theorem~\ref{main} follows from the propositions in~\cite[Section~8]{LS}, except when $L\cong G_2(q)$ with $q=2^f$ and $f$ odd, or $L\cong E_6(q)$ with $q=2^f$.
\item[\bf{Sporadic}]We deal with all sporadic simple groups.
\end{itemize}

Summing up, the remaining cases are almost simple groups having socle one of the following groups
$$\mathrm{PSL}_3(2^f) \,\,f\ge 3,\,\,\, G_2(2^f)\,\,f\ge3 \textrm{ odd},\,\,\, E_6(2^f),\,\,\, \textrm{sporadic}.$$

The proof of Corollary~1.4 in~\cite{LS} shows that, if $G$ is an almost simple group with socle $L$ and if $\Gamma_2(L)$ is strongly connected, then so is $\Gamma_2(G)$. In particular, in our arguments we may often assume $G=L$.

Most sporadic simple groups are dealt with in~\cite{DLN}. However, rather than pointing out the cases that need some extra consideration, we simply deal with the whole family. This will also help to clarify our strategy.

\section{Notation and preliminaries}\label{not&prel}
In what follows,  we write $x\mapsto_n y$ to denote the (directed) arc $(x,y)$ of $\Gamma_n(G)$.

To prove Theorem~\ref{main} we use various ingredients. The first is a relation among the Engel graph, the \textit{\textbf{prime graph}} and the \textit{\textbf{commuting graph}}. Given a finite group $G$, we denote by $\pi(G)$ the set of prime divisors of the order of $G$. More generally, given a positive integer $n$, we denote by $\pi(n)$ the set of prime divisors of $n$. Now, the prime graph $\Pi(G)$ of $G$ is the graph having vertex set $\pi(G)$ and where two distinct primes $r$ and $s$ are declared to be adjacent if and only if $G$ contains an element having order divisible by $rs$. 

Suppose now that ${\bf Z}(G)=1$. Observe that $\Gamma_1(G)$ is a subgraph of the Engel graph $\Gamma_n(G)$. In particular, the connected components of  $\Gamma_1(G)$ give useful information on the strongly connected components of $\Gamma_n(G)$. Now, a result of Williams~\cite{Williams} (see also~\cite[Theorem~4.4]{mp}) gives a method to control the connected components of $\Gamma_1(G)$ using the connected components of $\Pi(G)$.
\begin{theorem}[Theorem~4.4,~\cite{mp}]\label{theorem44}
Let $G$ be a  non-solvable group with ${\bf Z} (G ) = 1$, let $\Psi$ be a connected component of the commuting graph of $G$ and let $\psi:=\pi(\Psi)$ be the set of prime numbers dividing the order of some element of $\Psi$. Suppose $2\notin \psi$. Then $G$ has an abelian Hall $\psi$-subgroup $H$ which is isolated in the commuting graph of $G$ and $\Psi = H\setminus\{1\}$. In particular, $\Psi$ has diameter $1$ in the commuting graph and  $\psi$ is a connected component of the prime graph of $G$.
 \end{theorem}
 Theorem~\ref{theorem44} describes the connected components of the commuting graphs of non-solvable groups with trivial center.
They consist of perhaps more than one connected component containing involutions and  the
remaining connected components are complete graphs. Moreover, these remaining connected components determine (and are determined) by the connected components of the prime graph consisting of odd primes. In particular, the classification of Williams~\cite{Williams} of the connected components of the prime graph of simple groups is important in our work. We also use the work of Kondrat'ev and Mazurov~\cite{KoMa} on the prime graph of non-abelian simple groups, because it corrects some inaccuracy in~\cite{Williams}.

For simple groups of Lie type the work of Morgan and Parker~\cite{mp} simplifies further our analysis and allow us to deduce important informations on the connected components of $\Gamma_1(G)$ even when the prime graph $\Pi(G)$ is disconnected.
\begin{theorem}[Proposition~8.10,~\cite{mp}]\label{theorem111}
	Let $G$ be a simple group of Lie type and assume G is not isomorphic to one of
	the following:
 $$\mathrm{PSL}_3(4), \mathrm{PSL}_2(q), {}^2B_2(q), {}^2G_2(q), {}^2F_4(q) \hbox{ with }q\ge 8.$$ Then the commuting graph of $G$ has a connected component containing all elements of even order of $G$.
\end{theorem}
Combining Theorems~\ref{theorem44} and~\ref{theorem111}, we obtain this useful reduction.
\begin{corollary}[Corollary~2.4,~\cite{LS}]\label{cor}
Let $n$ be a positive integer, let $G$ be a simple group  with the property that the communting graph of $G$ has a unique connected component, $\Omega$ say, containing all elements having even order.  Then $\Gamma_n(G)$ is strongly connected if and only if, for every connected component $\psi$ of the prime graph of $G$ with $2\notin\psi$ and for every Hall $\psi$-subgroup $H$ of $G$ (whose existence is guaranteed by Theorem~$\ref{theorem44}$), there exists $h\in H\setminus \{1\}$ and $x,y\in\Omega$ with $x\mapsto_n h$ and $h\mapsto_n y$.
\end{corollary}
Observe that when $G$ is isomorphic to $\mathrm{PSL}_3(2^f)$ with $2^f\ge 8$, or $G_2(q)$, or $E_6(q)$, Theorem~$\ref{theorem111}$ guarantees the existence of $\Omega$ in Corollary~\ref{cor}. 

In this paper, we denote by ${\bf o}(g)$ the order of the group element $g\in G$ and by ${\bf N}_G(H)$ the normalizer of the subgroup $H\le G$. 

We conclude with a   basic observation.

\begin{lemma}\label{lemma:basic2}
Let $n$ be a positive integer, let $G$ be a group and let $\Omega$ be a strongly connected component of $\Gamma_n(G)$. Suppose that, for each vertex $g\notin \Omega$, there exists a subgroup $H$ of $G$ and $h_1,h_2\in H\cap \Omega$ with $g\in H$ and $h_1\mapsto_n g\mapsto_n h_2$. Then $\Gamma_n(G)$ is strongly connected.
\end{lemma}
Lemma~\ref{lemma:basic2} allows us to work recursively on the subgroups of $G$.

\section{The groups $\mathrm{PSL}_3(q)$, $G_2(q)$ and $E_6(q)$}\label{sec:linear}
Given a prime power $q$, we denote by $\mathbb{F}_q$ a finite field having cardinality $q$. We start with a preliminary observation.
\begin{lemma}\label{l:1}Let $q$ be an even prime power. Then there exists $a\in\mathbb{F}_q$ such that $T^3+aT+1$ is an irreducible polynomial over $\mathbb{F}_q$.
\end{lemma}
\begin{proof}
We argue by contradiction and we suppose that, for each $a\in\mathbb{F}_q$, the polynomial $f_a(T)=T^3+aT+1$ is reducible over $\mathbb{F}_q$. Since $f_a(T)$ is reducible, it has a root in $\mathbb{F}_q$. Suppose that $f_a(T)$ and $f_b(T)$ have a common root $r$. Then
$(f_a-f_b)(r)=0.$ As $f_a-f_b=(a-b)T$, we deduce that either $r=0$ or $a=b$. Since $f_a(0)=1$, we get $a=b$. This shows that the polynomials in $\{f_a\}_{a\in\mathbb{F}_q}$ have no common root.
Since the family $\{f_a\}_{a\in\mathbb{F}_q}$ consists of $q$ polynomials and since each polynomial in this family has a root in $\mathbb{F}_q$, we deduce that $f_a$ has a unique root $r_a\in\mathbb{F}_q$.

Let $q=2^f$. Observe that $f_1(T)=T^3+T+1$ is irreducible over $\mathbb{F}_2$. As $r_1\in \mathbb{F}_q$ and as $r_1$ is a root of $f_1$, we deduce  $\mathbb{F}_2[r_1]=\mathbb{F}_8\le\mathbb{F}_q$. Since $\mathbb{F}_8\le\mathbb{F}_q$, $f_1(T)$ has three roots in $\mathbb{F}_q$, contradicting the fact that $r_1$ is the only root of $f_1$ in $\mathbb{F}_q$.
\end{proof}
\begin{proposition}\label{propo:SL}Let $q\ge 4$ be an even prime power and let $G$ be an almost simple group having socle $L=\mathrm{PSL}_3(q)$. Then $\Gamma_2(G)$ is strongly connected, except when $G=\mathrm{PSL}_3(4)$. Moreover, $\Gamma_3(\mathrm{PSL}_3(4))$ is strongly connected. In particular, Theorem~$\ref{main}$ holds true for almost simple groups having socle $L$.
\end{proposition}
\begin{proof}
Suppose first  $G=L$ and $q\ge 8$. By Theorem~\ref{theorem111}, all elements of even order of $G$ are in the same connected component, $\Omega$ say, of $\Gamma_1(G)$. In particular, we use
the prime graph of $G$ to deduce properties of the commuting graph of $G$, see Corollary~\ref{cor}.  From~\cite{Williams}, $\Pi(G)$ has two connected components
$$\pi\left(q(q-1)(q^2-1) \right)\hbox{ and } \pi\left(\frac{q^2+q+1}{\gcd(3,q-1)}\right).$$ 

Using Lemma~\ref{l:1}, choose $a\in\mathbb{F}_q$ with $T^3+aT+1$ irreducible over $\mathbb{F}_q$. Now, consider the element
\[
g=
\left[
\begin{array}{ccc}
0&0&1\\
1&0&a\\
0&1&0
\end{array}
\right]\in\mathrm{PSL}_3(q).
\]
(We are using square brackets to denote elements in the projective special linear group.)
Observe that $g$ has order a divisor of $(q^2+q+1)/\gcd(3,q-1)$, because $T^3+aT+1$ is irreducible over $\mathbb{F}_q$ and because $g$ is the projection in $\mathrm{SL}_3(q)/{\bf Z}(\mathrm{SL}_3(q))=\mathrm{PSL}_3(q)$ of the companion matrix of $T^3+aT+1$.

Now, let $c$ be a generator of the multiplicative group of $\mathbb{F}_q$ and let
\[
h=
\left[
\begin{array}{ccc}
c&0&0\\
0&c&0\\
0&0&c^{-2}
\end{array}
\right]\in\mathrm{PSL}_3(q).
\]
Observe that $h\ne 1$. Indeed, $h=1$ implies $c=c^{-2}$, that is, $c^3=1$. Since $c$ has order $q-1$, we deduce $q-1\le 3$, contradicting the fact that we are assuming $q\ge 8$. A direct computations gives
$$[g,h]=
\left[
\begin{array}{ccc}
1&0&0\\
0&c^3&0\\
0&0&c^{-3}
\end{array}
\right].
$$
Since $[g,h]$ commutes with $h$, we deduce $[g,_2h]=1$.

We have shown that, for every $g\in G$ having order a divisor of $(q^2+q+1)/\gcd(3,q-1)$, there exists $h\in \Omega$ with $g\mapsto_2h$. Now, ${\bf N}_G(\langle g\rangle)=C\rtimes D$, where $C={\bf C}_G( g)$ is cyclic of order $(q^2+q+1)/\gcd(3,q-1)$ and $D$ is cyclic of order $3$. In particular, there exists $h\in {\bf N}_G(\langle g\rangle)$ with ${\bf o}(h)=3$. As $3\in \pi(q(q-1)(q^2-1))$, we have $h\in \Omega$. Moreover, $[h,g]\in \langle g\rangle$ and hence $[h,_2g]=1$. Thus $\Gamma_2(G)$ is strongly connected by Corollary~\ref{cor}.

When $L<G$ and $q\ge 8$, from the fact that $\Gamma_2(L)$ is strongly connected and from the proof of Corollary~1.4 in~\cite{LS}, it follows that $\Gamma_2(G)$ is strongly connected.

When $q=4$, we have checked the veracity of the statement with the invaluable help of a computer, using the computer algebra system magma~\cite{magma}.
\end{proof}
\begin{proposition}Let $q=2^f$ and let $G$ be an almost simple group having socle $L\in \{G_2(q),E_6(q)\}$. Suppose $f\ge 3$ is odd when $L=G_2(q)$. Then $\Gamma_2(G)$ is strongly connected.
\end{proposition}
\begin{proof}
Assume first $G=L$. In view of Corollary~\ref{cor}, we may assume that $\Pi(G)$ is disconnected. By Theorem~\ref{theorem111}, all elements of even order of $G$ are in the same connected component, $\Omega$ say, of $\Gamma_1(G)$.

Suppose  $G=L=E_6(q)$.  From~\cite{Williams}, $\Pi(G)$ has two connected components
$$\pi\left(q(q^5-1)(q^8-1)(q^{12}-1) \right)\hbox{ and } \pi\left(\frac{q^6+q^3+1}{\gcd(3,q-1)}\right).$$ 
From~\cite[Table~5.1]{LSS}, $G$ contains a reductive subgroup $H\cong\mathrm{PSL}_3(q^3)$. Let $g\in G$ having order a divisor of $(q^6+q^3+1)/\gcd(3,q-1)$. Now, $g$ is conjugate to an element of $H$. Therefore, replacing $g$ with this conjugate, we may suppose that $g\in H$. From Proposition~\ref{propo:SL}, there exists $h_1,h_2\in H\cap \Omega$ with $h_1\mapsto_2 g\mapsto_2 h_2$. Therefore $\Gamma_2(G)$ is strongly connected by Corollary~\ref{cor}.

Suppose next $G=L=G_2(q)$.  From~\cite{Williams}, $\Pi(G)$ has two connected components
$$\pi\left(q(q^2-1)(q^3+1) \right)\hbox{ and } \pi\left(q^2+q+1\right).$$ 
From~\cite[Table~8.30]{bray}, $G$ contains a maximal subgroup $H\cong\mathrm{SL}_3(q).2$. Let $g\in G$ having order a divisor of $q^2+q+1$. Now, $g$ is conjugate to an element of $H$. Therefore, replacing $g$ with this conjugate, we may suppose that $g\in H$. From Proposition~\ref{propo:SL}, there exists $h_1,h_2\in H\cap \Omega$ with $h_1\mapsto_2 g\mapsto_2 h_2$. Therefore $\Gamma_2(G)$ is strongly connected by Corollary~\ref{cor}.

When $L<G$, from the fact that $\Gamma_2(L)$ is strongly connected and from the proof of Corollary~1.4 in~\cite{LS},   $\Gamma_2(G)$ is strongly connected.
\end{proof}

\section{Sporadic groups}\label{sec:sporadic}
In this section we deal with all sporadic simple groups, except for the Monster group. For the rest of this paper, we use the notation from~\cite{atlas}.

Let $G$ be a sporadic simple group. Using~\cite[Lemma~$6.1$]{mp}, we see that the commuting graph of $G$ has a unique connected component, $\Omega$ say, containing all elements having even order. Let $\pi_1$ be the connected component of the prime graph $\Pi(G)$ with $2\in \pi_1$. Observe that  $\Omega$ contains all the elements whose order is divisible by a prime $p \in \pi_1$. 
	
	Except for the Monster group, where we need further considerations, we use two strategies to prove that $\Gamma_2(G)$ is strongly connected. First, we look at the prime graph $\Pi(G)$ and, for every prime $p\in \pi(G)\setminus \pi_1$, we consider an element $g \in G$ with ${\bf o}(g)=p$. Then we take a maximal subgroup $M \leq G$ containing $g$ (this can be  done using~\cite{atlas}). Then, two things can happen.
	\begin{enumerate}
		\item $\Gamma_2(M)$ is strongly connected. In this case, there exist $z_1,z_2 \in M \cap \Omega$ such that 
				$z_1 \rightarrow_2 g \rightarrow_2 z_2.$
		\item $\Gamma_2(M)$ is not strongly connected. In this case, we complete our argument with the aid of a  computer calculation. In all cases, we are able to find an involution $z$ of $G$ such that $g\rightarrow_2 z$ with a random search. On the other hand, using~\cite{atlas}, we see that in all cases there exists $z' \in \Omega \cap {\bf N}_G(\langle g \rangle)$. In particular, $[z',g] \in \langle g \rangle$, implying that $z'\mapsto_2g$.
	\end{enumerate}
	
We collect all the information we need to apply this strategy in Table~\ref{tableConnComp}. Indeed, for each of the 26 sporadic simple groups, we are recording in the second column the connected components of $\Pi(G)$ different from $\pi_1$. (We use~\cite{KoMa} for this information.) Moreover, for each prime $p$ in a connected component appearing in the second column, we are reporting (when it exists) in the third column a maximal subgroup $M$ of $G$ containing elements having order $p$ and with $\Gamma_2(M)$ strongly connected. 		
	
Now, for almost simple groups having socle a sporadic simple group different from the Monster, the proof of Theorem~\ref{main}  follows from Table~\ref{tableConnComp} and from a few computer computations. The only reason that the Monster group is an exception here, it is due to its enormous size, which makes it impossible to do any reasonable computation with it. Since the argument is very repetitive, we report here only the proof that $\Gamma_2(M_{22})$ and $\Gamma_2(M_{23})$ are strongly connected. We start with $G=M_{22}$. Observe that we only need to discuss the primes $p\in \{5,11\}$. For $p=5$, $G$ contains a maximal subgroup $M$ isomorphic to $2^4:\mathrm{Sym(5)}$. In particular, $G$ contains a subgroup isomorphic to $\mathrm{Sym}(5)$. Now, $\Gamma_2(\mathrm{Sym}(5))$ is strongly connected by our results on alternating groups. For $p=11$, $G$ contains a maximal subgroup $M$ isomorphic to $\mathrm{PSL}_2(11)$. Now, $\Gamma_2(\mathrm{PSL}_2(11))$ is strongly connected by our results on the projective special linear groups (observe that the largest power of $2$ dividing $(11+1)/2=6$ is $2$). 
Assume now $G=M_{23}$. Observe that we only need to discuss the primes $p\in \{11,23\}$. For $p=11$, $G$ contains a maximal subgroup $M$ isomorphic to $M_{22}$. Now, $\Gamma_2(M_{22})$ is strongly connected from above. When $p=23$, $G$ has no maximal subgroup $M$ with $\Gamma_2(M)$ strongly connected and hence with a computer computation we have found (by a random search) an involution $z$ with $g\mapsto_2z$. Therefore, also in this case, $\Gamma_2(G)$ is strongly connected.

	\begin{table}[!ht]\centering
		\begin{tabular}{ccl}
			\toprule[1.5pt]
			Group & Connected components & Maximal subgroups\\
			\midrule[1.5pt]
			$\mathrm{M}_{11}$& $\{5\}$&$\mathrm{S}_5$     \\ 
							& $\{11\}$ & $\mathrm{L}_2(11)$\\
			\hline 
			$\mathrm{M}_{12}$ & $\{11\}$ & $\mathrm{M}_{11}$ \\
			\hline 
			$\mathrm{M}_{22}$ & $\{5\}$  & $2^4:\mathrm{S}_5$				  \\
			                  & $\{7\}$  & $\mathrm{A}_7$ \\
			                  & $\{11\}$ & $\mathrm{L}_2(11)$ \\
			\hline 
			$\mathrm{M}_{23}$ & $\{11\}$ & $\mathrm{M}_{22}$ \\
							  & $\{23\}$ &  				\\
			\hline 
			$\mathrm{M}_{24}$ & $\{11\}$ & $\mathrm{M}_{23}$ \\
						      & $\{23\}$ & $\mathrm{M}_{23}$ \\
			\hline
			$\mathrm{J}_1$   & $\{7\}$  & 					\\
							 & $\{11\}$ & $\mathrm{L}_2(11)$ \\
							 & $\{19\}$ & 					\\
			\hline
			 $\mathrm{J}_2$  & $\{7\}$ &  $\mathrm{U}_3(3)$ \\
			\hline
			$\mathrm{J}_3$ & $\{17\}$ &  $\mathrm{L}_2(17)$ \\
						   & $\{19\}$ &  $\mathrm{L}_2(19)$ \\
			\hline
			$\mathrm{J}_4$ & $\{23\}$ & $ 2^{11}:\mathrm{M}_{23}$ \\
						   & $\{29\}$ & 					\\
						   & $\{31\}$ & $\mathrm{L}_2(32):2$ \\
						   & $\{37\}$ & \\
						   & $\{43\}$ & \\
			\hline 
			$\mathrm{HS}$  & $\{7\}$  & $\mathrm{M}_{22}$ \\
						   & $\{11\}$ & $\mathrm{M}_{22}$ \\
			\hline
			$\mathrm{McL}$ & $\{11\}$ & $\mathrm{M}_{11}$ \\
			\hline 
			$\mathrm{He}$  & $\{17\}$ & $\mathrm{S}_4(4):2$   \\
			\hline
			$\mathrm{Ru}$  & $\{29\}$ & 			\\
			\hline
			$\mathrm{Suz}$ & $\{11\}$ & $\mathrm{U}_5(2)$ \\
						   & $\{13\}$ & $\mathrm{G}_2(4)$ \\
			\hline
			$\mathrm{O'N}$ &$\{11\}$  & $\mathrm{M}_{11}$ \\
						   &$\{19\}$  & $\mathrm{L}_3(7):2$ \\
						   &$\{31\}$  &			\\
			\hline
			$\mathrm{Co}_3$&$\{23\}$  & $\mathrm{M}_{23}$ \\
			\hline
			$\mathrm{Co}_2$&$\{11\}$  & $\mathrm{HS}$ \\
						   &$\{23\}$  & $\mathrm{M}_{23}$ \\
			\hline
			$\mathrm{Fi}_{22}$&$\{13\}$ & $\mathrm{O}_7(3)$ \\
			\hline
			$\mathrm{HN}$  & $\{19\}$ & $\mathrm{U}_3(8):3$ \\
			\hline
			$\mathrm{Ly}$  & $\{31\}$ & $\mathrm{G}_2(5)$ \\
						   & $\{37\}$ & 			\\
						   & $\{67\}$ & 			\\
			\hline
			$\mathrm{Th}$  & $\{19\}$ & $\mathrm{U}_3(8):6$ \\
						   & $\{31\}$ & ${}^3\mathrm{D}_4(2):3$ \\
			\hline 
			$\mathrm{Fi}_{23}$&$\{17\}$ & $\mathrm{S}_8(2)$ \\
						    &$\{23\}$ & $2^{11}:\mathrm{M}_{23}$\\
			\hline
			$\mathrm{Co}_1$&$\{23\}$ & $\mathrm{Co}_2$ \\
			\hline
			$\mathrm{Fi}'_{24}$& $\{17\}$ & $\mathrm{Fi}_{23}$ \\
							& $\{23\}$  & $\mathrm{Fi}_{23}$ \\
							& $\{29\}$ & \\
			\hline
			$\mathrm{B}$ & $\{31\}$ & $\mathrm{Th}$ \\
						 & $\{47\}$ & \\
			\hline
			$\mathrm{M}$ & $\{47\}$ & $3^8.\mathrm{O}_8^{-}(3).2_3$ \\
						 & $\{59\}$ & $\mathrm{L}_2(59)$ \\
						 & $\{71\}$ & \\
			\bottomrule[1.5pt]
		\end{tabular}
		\caption{Sporadic groups and connected components of the prime graph not containing the prime $2$. Notation from~\cite{atlas}.}\label{tableConnComp}
	\end{table}
\section{The Monster group}
Let $G$ be the Monster group. From Table~\ref{tableConnComp} and from the arguments (and the notation) in the previous sections, to prove Theorem~\ref{main} it suffices to show that, given an element $x\in G$ with ${\bf o}(x)=71$, there exists $z\in \Omega$ with $x\mapsto_2 z$.

Let $\iota\in G$ be an involution in the conjugacy class $2A$ of $G$ and let $H={\bf C}_G(\iota)$. Now, $H$ is a maximal subgroup of $G$ and $H\cong 2.B$, where $B$ is the Baby Monster. Let $1_H^G$ be the permutation character for the action of $G$ on the set $\mathcal{I}$ of right cosets of $H$ in $G$. From~\cite[page~115]{MuSp}, we see that in this action $G$ has rank $9$ and, in the decomposition of $\pi$ as the sum of complex irreducible characters, we have
$$\pi=\chi_1+\chi_2+\chi_4+\chi_5+\chi_9+\chi_{14}+\chi_{21}+\chi_{34}+\chi_{35},$$
where we are labeling the irreducible complex characters of $G$ consistently with the notation in~\cite{atlas}.

As $H$ is the centralizer of the involution $\iota$, we may also identify $\mathcal{I}$ with the $G$-conjugacy class $\iota^G$. Using the notation in~\cite{atlas}, the conjugacy class $\mathcal{I}$ is $2A$. From~\cite{norton}, the subdegrees for the action of $G$ on $\mathcal{I}$ are
\begin{align*}
&i_0=1,\,
i_1=27143910000,\,
i_2=11707448673375,\,
i_3=2031941058560000,\,
i_4=91569524834304000,\\
i_5=&1102935324621312000,\,
i_6=1254793905192960000,\,
i_7=30434513446055706624,\,
i_8=64353605265653760000.
\end{align*}
We let $\mathcal{I}_0,\ldots,\mathcal{I}_8$, be the orbits of $H$ on $\mathcal{I}$ with $i_j=|\mathcal{I}_j|$. In particular, $\mathcal{I}_0=\{\iota\}$.

Let $x\in G$ be an element of order $71$, let $C=\langle x\rangle$ and let $$\mathcal{C}=\bigcup_{c\in C\setminus \{1\}}c^G.$$ We observe that $H,C$ and $\mathcal{C}$ satisfy Hypothesis~0-3 in~\cite[Section~8.1]{LS}. Indeed,
\begin{description}
\item[Hypothesis~0] $H\cap \mathcal{C}=\emptyset$, because the elements in $\mathcal{C}$ have order $71$ and the order of $H$ is 
relatively prime to $71$,
\item[Hypothesis~1]$\mathcal{C}=\{c^g\mid c\in C\setminus\{1\},g\in G\}$, directly from the definition of $\mathcal{C}$,
\item[Hypothesis~2]for each $g\in G\setminus{\bf N}_G(C)$, we have $C^g\cap C=1$, because $C$ has prime order,
\item[Hypothesis~3]${\bf N}_G(C)$ is a Frobenius group with Frobenius kernel $71$, indeed, $C$ is contained in a unique $G$-conjugacy class of maximal subgroups of $G$ and these maximal subgroups are isomorphic to $\mathrm{PSL}_2(71)$. Thus, $C\le {\bf N}_G(C)\le K$, with $K\cong \mathrm{PSL}_2(71)$. Therefore, ${\bf N}_G(C)={\bf N}_K(C)\cong 71:35$ is a Frobenius group having kernel $C$.
\end{description}

\begin{lemma}\label{lemma:monster}There exists an involution $z\in G$ with $[x,_2z]=1$. In particular, $\Gamma_2(G)$ is strongly connected. 
\end{lemma}
\begin{proof}
We start by proving an auxiliary result concerning $\iota^{\mathcal{C}}=\{\iota^g\mid g\in \mathcal{C}\}$. We claim that
\begin{itemize}
\item either $\iota^{\mathcal{C}}=\mathcal{I}\setminus\{\iota\}$, 
\item or $\iota^{\mathcal{C}}=\mathcal{I}\setminus(\{\iota\}\cup \mathcal{I}_1)$.
\end{itemize}
We are not able to resolve the ambiguity between these two cases, but this is not a concern for our application. 
\\
Observe that $H\mathcal{C}=\{hg\mid h\in H,g\in \mathcal{C}\}$ is a union of right $H$-cosets. We write $|H\mathcal{C}:H|$ to denote the number of right $H$-cosets in $H\mathcal{C}$. 
We claim that $H\mathcal{C}=H\mathcal{C}H$. Indeed, $H\mathcal{C}\subseteq H\mathcal{C}H$. Conversely, if $h_1ch_2\in H\mathcal{C}H$ with $h_1,h_2\in H$ and $c\in\mathcal{C}$, then $h_1ch_2=h_1h_2c^{h_2}\in H\mathcal{C}$. This shows that $H\mathcal{C}$ is a union of $H$ right cosets.
\\
 As $H={\bf C}_G(\iota)$,  we have 
$|H\mathcal{C}:H|=|\{\iota^c\mid c\in\mathcal{C}\}|.$ Moreover, since $H\mathcal{C}=H\mathcal{C}H$, the set $\{\iota^c\mid c\in\mathcal{C}\}$ is a union of $H$-orbits for the action of $H$ on $\mathcal{I}$. Therefore, to prove our claim, it suffices to show that
$$|H\mathcal{C}:H|\in \{|G:H|-1,|G:H|-1-i_1\}.$$ 

As in~\cite[Section~8.1]{LS}, we let 
$$\Delta(H,\mathcal{C})$$
be the auxiliary graph having vertex set $\mathcal{C}$ and where two vertices $x,y\in\mathcal{C}$ are declared to be adjacent if and only if $yx^{-1}\in H$. Observe that $\Delta(H,\mathcal{C})$ is undirected, because if $yx^{-1}\in H$, then $xy^{-1}=(yx^{-1})^{-1}\in H$.

We now describe a relation between the set $H\mathcal{C}$ and the connected components of $\Delta(H,\mathcal{C})$.
Let $\Delta_1,\ldots,\Delta_c$ be the connected components of $\Delta(H,\mathcal{C})$. For each $i\in \{1,\ldots,c\}$, let $x_i\in \Delta_i$. We claim that $$\Delta_i\subseteq Hx_i.$$ Indeed, for this end, let $y\in \Delta_i$. Since $\Delta_i$ is connected, there exists a path $x_i=p_1,p_2,\ldots,p_\ell=y$ in $\Delta_i$ with $p_j$ adjacent to $p_{j+1}$ for every $j\in \{1,\ldots,\ell-1\}$. Now, $p_2x_i^{-1}=p_2p_1^{-1}\in H$ and hence $Hx_i=Hp_2$. Therefore, arguing inductively on $\ell$, we have $Hx_i=Hy$ and so $y\in Hx_i$.

The cosets $Hx_1,\ldots,Hx_c$ are distinct. In fact, if $Hx_i=Hx_j$, then $x_ix_j^{-1}\in H$ and hence $x_i$ is adjacent to $x_j$ in $\Delta(H,\mathcal{C})$. Since $x_i\in \Delta_i$ and $x_j\in \Delta_j$ and since $\Delta_i$ and $\Delta_j$ are connected components of $\Delta(H,\mathcal{C})$, we get $\Delta_i=\Delta_j$ and $x_i=x_j$. 

Summing up, we have
$$H\mathcal{C}=Hx_1\cup\cdots \cup Hx_c$$
and hence $|H\mathcal{C}:H|=c$ is the number of connected components of $\Delta(H,\mathcal{C})$.

Section~8.1 in~\cite{LS} develops a method for giving a lower bound on $c$ using only the character table of the group $G$ and the decomposition of the permutation character $\pi=1_H^G$. Indeed,~\cite[(8.4)]{LS} gives
$$c\ge |G:H|\cdot \frac{(|C|-1)^2}{\sum_{\chi\in\mathrm{Irr}(G)}\frac{\langle\chi,\pi\rangle}{\chi(1)}(|C|\langle\chi_{|C},1\rangle-\chi(1))^2}.$$ 
In our case, using the decomposition of $\pi$ as the sum of irreducible characters and using the character table of $G$, we find
\begin{align}\label{align1}c\ge \frac{4020948952911201379068000}{41351}.\end{align}
Let $c'$ be the right hand side of this inequality. Now, if $c=|G:H|-1$, then $\iota^{\mathcal{C}}=\mathcal{I}\setminus\{\iota\}$. Therefore, suppose $c<|G:H|-1$.

 Therefore, $c$ is a sum of nontrivial subdegrees for the action of $G$ on the right cosets of $H$ in $G$, that is, for the action of $G$ on $\mathcal{I}$. However, as $$c' > |G:H|-1-i_2,$$ we obtain
$$c=|G:H|-1-27143910000=|G:H|-1-i_1$$
and hence $\iota^{\mathcal{C}}=\mathcal{I}\setminus (\{\iota\}\cup\mathcal{I}_1)$. This concludes the proof of our claim.

\medskip

Let $\mathcal{X}_1,\ldots,\mathcal{X}_{194}$ be the $G$-conjugacy classes and, for each $i$, let $x_i\in \mathcal{X}_i$. For any $G$-conjugacy class $\mathcal{X}_i$, we define
$$\overline{\mathcal{X}}_i=\sum_{g\in\mathcal{X}_i}g\in\mathbb{C}G,$$
where $\mathbb{C}G$ is the complex group algebra over $G$. More generally, for every subset $X$ of $G$, we define $\overline{X}=\sum_{x\in X}x\in\mathbb{C}[G]$. Then
$$\overline{\mathcal{X}}_i\cdot\overline{\mathcal{X}_j}=\sum_{v=1}^{194}a_{ijv}\overline{\mathcal{X}}_{v},$$
where $a_{ijv}\in \mathbb{N}$ are the class constants of $G$. There is a combinatorial interpretation of the $a_{ijv}$, which comes from the conjugacy class association scheme of $G$. Indeed,
$$a_{ijv}=|\{(a,b)\in\mathcal{X}_i\times\mathcal{X}_j\mid ab=x_v\}|.$$
The class constants can be computed using the character table of $G$. From~\cite[(3.9)]{Isaacs}, we have
$$\frac{|G|}{|\mathcal{X}_i||\mathcal{X}_j|}a_{ijv}=\sum_{\chi\in\mathrm{Irr}(G)}\frac{\chi(x_i)\chi(x_j)\chi(x_v^{-1})}{\chi(1)}.$$
Using the character table for $G$, we obtain
\begin{align}\label{eq:classconstants0}
\overline{\mathcal{I}}\cdot\overline{\mathcal{I}}&=|\mathcal{I}|\cdot\overline{1A}+27143910000\cdot\overline{2A}+196560\cdot\overline{2B}+920808\cdot\overline{3A}+3\cdot\overline{3C}+1104\cdot\overline{4A}+4\cdot\overline{4B}+5\cdot\overline{5A}+6\cdot\overline{6A},
\end{align}
where we are using the same labeling as in~\cite{atlas}. As $\iota\in\mathcal{I}$, from this it follows that
\begin{align}\label{eq:classconstants}
\iota\cdot\overline{\mathcal{I}}&=\overline{1A}+\overline{X}_{2A}+\overline{X}_{2B}+\overline{X}_{3A}+\overline{X}_{3C}+\overline{X}_{4A}+\overline{X}_{2B}+\overline{X}_{5A}+\overline{X}_{6A},
\end{align}
where $X_{2A}\subseteq 2A$,  $X_{2B}\subseteq 2B$, $X_{3A}\subseteq 3A$, $X_{3C}\subseteq 3C$, $X_{4A}\subseteq 4A$,  $X_{4B}\subseteq 4B$,  $X_{5A}\subseteq 5A$,  $X_{6A}\subseteq 6A$. 

Let $Y\in \{1A,2A,2B,3A,3C,4A,4B,5A,6A\}$ and let $m_Y\in\mathbb{N}$ be the class constant appearing in~\eqref{eq:classconstants0} in front of $\overline{Y}$. We claim that $|X_Y|=m_Y|Y|/|\mathcal{I}|$ (for instance, $|X_{4A}|=1104|4A|/|\mathcal{I}|$ and $|X_{6A}|=6|6A|/|\mathcal{I}|$). To this end, we double count the elements in the set $\{(a,b)\in\mathcal{I}\times \mathcal{I}\mid ab\in Y\}$. On one side, we have
$$|\{(a,b)\in\mathcal{I}\times \mathcal{I}\mid ab\in Y\}|=\sum_{y\in Y}|\{(a,b)\in\mathcal{I}\times \mathcal{I}\mid ab=y\}|=\sum_{y\in Y}m_Y=m_Y|Y|.$$
On the other side, let $a,a'\in\mathcal{I}$. Since $a,a'$ are in the same $G$-conjugacy class, there exists $g\in G$ with $a'=a^g$. In particular, 
$$\{b\in \mathcal{I}\mid ab\in Y\}^g=\{b^g\in\mathcal{I}^g\mid a^gb^g\in Y^g\}=\{b^g\in\mathcal{I}\mid a'b^g\in Y\}=\{b\in\mathcal{I}\mid a'b\in Y\}.$$
Therefore, applying this fact with $a=\iota$, we obtain
$$|\{(a,b)\in\mathcal{I}\times \mathcal{I}\mid ab\in Y\}|=|\mathcal{I}||\{b\in \mathcal{I}\mid \iota b\in Y\}|=|\mathcal{I}||X_Y|.$$
Thus $m_Y|Y|=|\mathcal{I}||X_Y|$ and hence $|X_Y|=m_Y|Y|/|\mathcal{I}|$.

From the previous paragraph and~\eqref{eq:classconstants}, we obtain
\begin{align*}%\label{eq:classconstants2}
&|X_{2A}|=27143910000,&&|X_{2B}|=11707448673375,\\
 &|X_{3A}|=2031941058560000,&&|X_{3C}|=91569524834304000,\\
&|X_{4A}|=1102935324621312000, &&|X_{4B}|=1254793905192960000,\\
 &|X_{5A}|=30434513446055706624,&&|X_{6A}|=64353605265653760000.
\end{align*}

Now, let $$\overline{\iota^{\mathcal{C}}}=\sum_{x\in \iota^{\mathcal{C}}}x\in\mathbb{C}[G].$$
As $\iota^{\mathcal{C}}\subseteq\mathcal{I}$, from~\eqref{eq:classconstants}, we deduce 
$$\iota\cdot\overline{\mathcal{\iota^{\mathcal{C}}}}=
\overline{X'}_{2A}+\overline{X'}_{2B}+\overline{X'}_{3A}+\overline{X'}_{3C}+\overline{X'}_{4A}+\overline{X'}_{4B}+\overline{X'}_{5A}+\overline{X'}_{6A},$$
with $X'_{Y}\subseteq X_Y$ for every conjugacy class $Y$,
and
$$|\iota^{\mathcal{C}}|=|X'_{2A}|+|X'_{2B}|+|X'_{3A}|+|X'_{3C}|+|X'_{4A}|+|X'_{4B}|+|X'_{5A}|+|X'_{6A}|.$$

Suppose $X'_{2A}\ne \emptyset$ or $X'_{2B}\ne \emptyset$. In particular, there exists $g\in\mathcal{C}$ with $\iota\iota^{g}\in X'_{2A}\subseteq 2A$ when $X'_{2A}\ne \emptyset$ and $\iota\iota^{g}\in X'_{2B}\subseteq 2B$ when $X'_{2B}\ne \emptyset$. 
This shows that there exists $g\in \mathcal{C}$ such that $\iota \iota^g$ is an involution. Hence $\iota\iota^g=\iota^g\iota=[g,\iota]$ is an involution.  Therefore, $[g,\iota,\iota]=\iota^g\iota\cdot \iota\cdot \iota^g\iota\cdot \iota=1$, that is, $g\mapsto_2 \iota$. Since $\langle x\rangle$ and $\langle g\rangle$ are Sylow $71$-subgroups of $G$, there exists $m\in G$ and $i\in\mathbb{Z}$ with $x^i=g^m$. In particular, $x^i\mapsto_2 z$, where $z=\iota^m$ is an involution.

Suppose $X'_{2A}=X'_{2B}=\emptyset$.
From our earlier claim, we have
$\overline{\mathcal{I}}=\overline{\iota^{\mathcal{C}}}+\iota$ or $\overline{\mathcal{I}}=\overline{\iota^{\mathcal{C}}}+\overline{\mathcal{I}}_1+\iota.$ Assume first that $\overline{\mathcal{I}}=\overline{\iota^{\mathcal{C}}}+\iota$. Then
$$\iota\cdot\overline{\iota^{\mathcal{C}}}=\iota\cdot (\overline{\mathcal{I}}-\iota)=\iota\cdot\overline{\mathcal{I}}-1A=\overline{X}_{2A}+\overline{X}_{2B}+\overline{X}_{3A}+\overline{X}_{3C}+\overline{X}_{4A}+\overline{X}_{2B}+\overline{X}_{5A}+\overline{X}_{6A}.$$
In particular, $X'_{2A}=X_{2A}\ne\emptyset$ and $X'_{2B}=X_{2B}\ne\emptyset$, which contradicts $X'_{2A}=X'_{2B}=\emptyset$. Assume next that $\overline{\mathcal{I}}=\overline{\iota^{\mathcal{C}}}+\overline{\mathcal{I}_1}+\iota$. Then
\begin{align*}
|\iota\cdot\iota^{\mathcal{C}}|&=|\mathcal{I}\setminus(\mathcal{I}_1\cup\{\iota\})|=|\mathcal{I}|-1-i_1=97239461114865275999.
\end{align*}
On the other hand, since $X'_{2A}=X'_{2B}=\emptyset$, we have
\begin{align*}
|\iota\cdot \iota^{\mathcal{C}}|&=|X'_{3A}|+|X'_{3C}|+|X'_{4A}|+|X'_{2B}|+|X'_{5A}|+|X'_{6A}|\\
&\le |X_{3A}|+|X_{3C}|+|X_{4A}|+|X_{4B}|+|X_{5A}|+|X_{6A}|=97239449407416602624.
\end{align*}
However, this is a contradiction because
$$97239461114865275999>97239449407416602624.\qedhere$$
\end{proof}
%\begin{table}[!ht]
%\begin{tabular}{c|c|c}\hline
%Group&Nr&Connected components\\\hline
%$M_{22}$&4&$\{2,3\}$, $\{5\}$, $\{7\}$, $\{11\}$\\
%$M_{22}.2$&2&$\{2,3,5,7\}$, $\{11\}$\\
%$J_{1}$&4&$\{2,3,5\}$, $\{7\}$, $\{11\}$, $\{19\}$\\
%$O'N$&4&$\{2,3,5,7\}$, $\{11\}$, $\{19\}$, $\{31\}$\\
%$O'N.2$&2&$\{2,3,5,7,11,19\}$, $\{31\}$\\
%$Ly$&4&$\{2,3,5,7,11\}$, $\{31\}$, $\{37\}$, $\{67\}$\\
%$Fi_{22}'$&4&$\{2,3,5,7,11,13\}$, $\{17\}$, $\{23\}$, $\{29\}$\\
%$Fi_{22}$&4&$\{2,3,5,7,11,13\}$, $\{17\}$, $\{23\}$, $\{29\}$\\
%$M$&$4$&$\{2,3,5,7,11,13,17,19,23,29,31,47\}$, $\{41\}$, $\{59\}$, $\{71\}$\\
%$J_4$&6&$\{2,3,5,7,11\}$, $\{23\}$, $\{29\}$, $\{31\}$, $\{37\}$, $\{43\}$
%\end{tabular}
%\caption{Connected components of the disconnected prime graph of almost simple groups having sporadic socle}\label{table:sporadic}
%\end{table}
\end{document}